\documentclass[bj]{imsart}

\RequirePackage{amsthm,amsmath,amsfonts,amssymb}
\RequirePackage[numbers]{natbib}
\RequirePackage[colorlinks,citecolor=blue,urlcolor=blue]{hyperref}
\RequirePackage{graphicx}

\startlocaldefs

\newtheorem{theorem}{Theorem}[section]
\newtheorem{lemma}[theorem]{Lemma}
\newtheorem{remark}{Remark}[section]

\newtheorem{corollary}{Corollary}[section]

\renewcommand{\(}{$\,}
\renewcommand{\)}{\,$}

\newcommand{\normx}[1]{\left \|#1 \right \|}
\newcommand{\normX}[1]{\left \|#1 \right \|}

\def\rb0{r_0^{\flat}}

\def\xiv{\bf{\xi}}

\def\ND{\mathcal{N}}

\newcommand{\bb}[1]{\boldsymbol{#1}}

\renewcommand{\Gamma}{\varGamma}
\renewcommand{\Pi}{\varPi}
\renewcommand{\Sigma}{\varSigma}
\renewcommand{\Delta}{\varDelta}
\renewcommand{\Lambda}{\varLambda}
\renewcommand{\Psi}{\varPsi}
\renewcommand{\Phi}{\varPhi}
\renewcommand{\Theta}{\varTheta}
\renewcommand{\Omega}{\varOmega}
\renewcommand{\Xi}{\varXi}
\renewcommand{\Upsilon}{\varUpsilon}

\def\argmax{\operatornamewithlimits{argmax}}

\def\R{I\!\!R}
\def\E{I\!\!E}
\def\P{I\!\!P}

\def\kappa{\varkappa}

\def\T{\top}

\def\gammav{\bb{\gamma}}

\def\gammav{\bb{\gamma}}

\def\tauv{\bb{\tau}}

\def\xiv{\bb{\xi}}
\def\zetav{\bb{\zeta}}

\def\dLb12{T_h^{\flat}(\theta_1^{\flat}, \theta_2^{\flat})}

\def\alphab{\alpha^{\flat}}

\def\alphab12{\alpha^{\flat}(\theta, \theta_0)}
\def\chib12{\chi^{\flat}(\theta, \theta_0)}

\def\Lb0{L^{\flat}(\theta_0)}

\def\L0{L(\theta_0)}

\newcommand{\vertiii}[1]{{\left\vert\kern-0.25ex\left\vert\kern-0.25ex\left\vert #1 
    \right\vert\kern-0.25ex\right\vert\kern-0.25ex\right\vert}}

\def\Ind{\operatorname{1}\hspace{-4.3pt}\operatorname{I}}

\def\ex{\mathrm{e}}

\def\gm{\mathtt{g}}

\def\xx{\mathtt{x}}
\def\xxc{\xx_{c}}

\def\nunu{\nu_{0}}

\def\dimp{p}

\def\N{\mathbb{N}}

\def\lambdav{\bb{\lambda}}

\def\zq{z}

\theoremstyle{remark}

\endlocaldefs

\begin{document}

\begin{frontmatter}
\title{Strong Gaussian Approximation for the Sum of Random Vectors}
\runtitle{Strong Gaussian Approximation}

\begin{aug}
\author[A]{\fnms{Nazar} \snm{Buzun}\ead[label=e1]{\{n.buzun}},
\author[A]{\fnms{Nikolay} \snm{Shvetsov}\ead[label=e2]{nikolay.shvetsov}}
\and
\author[A]{\fnms{Dmitry V.} \snm{Dylov}\ead[label=e3]{d.dylov\}@skoltech.ru}}
\address[A]{
  Skolkovo Institute of Science and Technology\\
  Bolshoi blvd. 30/1, Moscow, 121205, Russia \\
\printead{e1,e2,e3}}

\end{aug}

\begin{abstract}
This paper derives a new strong Gaussian approximation bound for the sum of independent random vectors. The approach relies on the optimal transport theory and yields \textit{explicit} dependence on the dimension size $p$ and the sample size $n$. This dependence establishes a new fundamental limit for all practical applications of statistical learning theory. Particularly, based on this bound, we prove approximation in distribution for the maximum norm in a high-dimensional setting ($p >n$).   
\end{abstract}


\end{frontmatter}

\section{Introduction}
Arguably, accurate estimation of a strong approximation is much more interpretable and meaningful than accurate estimation of the proximity of distributions.
In the latter case, one requires
a joint distribution of the original sum and the Gaussian vector to be built in one probabilistic space so that their proximity to each other could be gauged.
Gaussian approximation for the sum of random vectors attracts the attention of the mathematicians because of the uncertain dependence of the outcome on the dimension size $p$ \cite{Zai13,Bentkus,Chernozhukov} and is one of the most important tasks in the field of limit theorems of the probability theory.
In \citep{Zai13}, the uncertain property of the approximation was discerned, yielding the following finite-sample bound: 
\begin{equation} \label{zai_gar}
\P \left( \sup_{1 \leq h \leq n} \left \|  \sum_{i=1}^h \xiv_{i} - \sum_{i=1}^h \gammav_i \right\| > C_1(\alpha) p^{\frac{23}{4} + \alpha} \log p \log n + t \, C_2 p^{\frac{7}{2}} \log p \right) \leq
e^{-t},
\end{equation}
under conditions that the vectors $\{\xiv_i\}_{i=1}^n$ are centered and independent, $\sum_{i=1}^n \xiv_{i}$ and $ \sum_{i=1}^n \gammav_i \in \ND(0, \Sigma)$ have the same variance matrix and finite exponential moments over the same domain in $\R^p$.
The handicap of this result is the power of $p$, which requires extremely large sample sizes $n$ for the approximation to be of any practical value. Asymptotically it is required $n > p^{11}$ samples.    

On the other hand, the multivariate Gaussian approximation in distribution in the general case, studied in \cite{Bentkus}, has been known to yield an alternative error bound that favors smaller sample size. Namely, for all convex events $\mathcal{A}$:
\begin{equation}
\left |
\P \left( \sum_{i=1}^n \xiv_{i} \in \mathcal{A} \right) - \P \left( \sum_{i=1}^n \gammav_{i} \in \mathcal{A} \right)
\right| \leq O\left(\frac{p^{7/4}}{\sqrt{n}} \right).
\label{bentkus}
\end{equation}
For a particular case of Euclidean norm, it was proven in \cite{buzun2019gaussian} that the upper bound approaches the asymptotic value~$p / \sqrt{n}$. 
Naturally, the observed gap in inequalities (\ref{zai_gar}) and (\ref{bentkus}) suggests that the dependence on $p$ in the strong Gaussian approximation may be improved. 

In this paper, we narrow this gap by deriving a new type of strong Gaussian approximation bound for the sum of independent random vectors. Our approach is based on the optimal transport theory and the Gaussian approximation in the Wasserstein distance. 
We derive the approximation in probability in a simplified form without the maximum by~$h$ (refer to expression (\ref{zai_gar})) and our convergence rate is comparable with the approximation in distribution (Refs. \cite{Bentkus, buzun2019gaussian}). 
Specifically, given the sub-Gaussian assumption for vectors  $\{\xiv_i\}_{i=1}^n$  we prove that 
\[
\P \left(  \left \|  \sum_{i=1}^n \xiv_{i} - \sum_{i=1}^n \gammav_i \right\| > O \left (   \frac{ p^{3/2} \log p \log n  }{\sqrt{n}} \right) \, e^{t/\log(np)} \right) \leq
e^{-t}.
\]

\paragraph*{Related work.}  
As already mentioned, the authors of \cite{Zai13} improved the accuracy of strong Gaussian approximation for the sum of random independent vectors. 
In \cite{erd1946certain}, it was observed that the limit distributions of some functional of the growing sums of independent identically distributed random variables (with a finite variance) do not depend on the distribution of individual terms and, therefore, the approximation can be computed if the distribution of the terms has a specific simple form. 
In \cite{einmahl1989extensions}, the work reports new multidimensional results for the accuracy of the strong Gaussian approximation for infinite sequences of sums of independent random vectors. 
Consequently, Gaussian approximation of the sums of independent random vectors with finite moments \cite{Zai07}, including multidimensional approaches \cite{zaitsev2001multidimensional,saharenko2006,gotze2009bounds}, have been reported. 

Notably, the strong Gaussian approximation can help approximate the maximum sum of random vectors in distribution for the high-dimensional case $(p > n)$.  
Gaussian approximation of the maximum function is very useful for justifying the Bootstrap validity and for approximating the distributions with different statistics in high-dimensional models. Besides, the aforementioned papers \cite{Chernozhukov, chernozhukov2017detailed}, some relevant results can also be found in works \cite{koike2019gaussian} and \cite{SUN2020108523}, where the authors rely on Malliavin calculus and high-order moments to assess the corresponding bounds. In \cite{fang2020highdimensional}, the authors go after the same approximation as reported herein; however, using a completely different tool: the Stein method instead of the Wasserstein distance, yielding a result that is only valid under the constraint that the measure of i.i.d. vectors $\xiv_i$ has the Stein kernel (a consequence of the log-concavity).

In the most recent works \cite{chernozhukov2019improved}, the result of \cite{Chernozhukov} was superseded by considering specific distribution of the max statistic in high dimensions. 
This statistic takes the form of the maximum over the components of the sum of independent random vectors and its distribution plays a key role in many high-dimensional econometric problems. 
The new iterative randomized Lindeberg method allowed the authors to derive new bounds for the
distributional approximation errors. 
Specifically, in \cite{chernozhukov2020nearly}, new nearly optimal bounds ($\log^{3/2} p / \sqrt{n}$) for the Gaussian approximation  over the class of rectangles were obtained, functional in the case when the covariance matrix of the scaled average is non-degenerate. The authors also demonstrated that the bounds can be further improved in some special smooth and zero-skewness cases.

The demand for strong Gaussian approximation can frequently be encountered in a variety of statistical learning problems, being very important for developing efficient approximation algorithms. Likewise, strong Gaussian approximation is sought after in many applications across different domains (physics, biology, economy, \textit{etc}.), making it one of the most desired tools of modern probability theory.
Practical demand for the approximation are widespread and range from change point detection \cite{cpd_boot}, to variance matrix estimation \cite{avanesov2018}, to selection of high-dimensional sparse regression models \cite{Chernozhukov_2013}, to adaptive specification testing \cite{Chernozhukov_2013}, to anomaly detection in periodic signals \cite{cpd_ecg}, and to many others.

\section{Main result}

 In this work, we prove the Theorem that finds the upper bound for the strong Gaussian approximation. Herein, we consider a sum of independent zero-mean random vectors $\xiv = \sum_{i=1}^n \xiv_i$ in $\R^{p}$ that has a covariance matrix 
$$
\Sigma = \E \xiv \xiv^T.
$$
A Gaussian random vector $\gammav \in \ND(0, \Sigma)$ has the same 1-st and the 2-nd moments. 
 We say that a vector \(\xiv\) has restricted exponential or sub-Gaussian moments if $ \exists \; \gm>0$
\begin{equation}
    \log \E \exp\bigl( \lambdav^{\T}\xiv\bigr)
    \le
     \| \lambdav \|^{2}/2,
    \quad
    \forall \lambdav \in \R^{\dimp}, \quad \| \lambdav \| \le \gm 
\label{expgamgm_2}.
\end{equation}

\begin{theorem} \label{main_res} Under condition that $ \forall i: \, \frac{\sqrt{n}}{\nunu} \Sigma^{-1/2} \xiv_i $ is sub-Gaussian (expression (\ref{expgamgm_2}) with \( 0.3 \gm \ge \sqrt{\dimp} \)) exists a Gaussian vector $ \gammav \in \ND (0, \Sigma)$ such that 

\[
\P \left(  \left \|   \xiv -  \gammav \right\| > C(n, p) \|\Sigma \|^{1/2} \, e^{t/\log(np)} \right) \leq
e^{-t},
\]
where 
\[
C(n, p) =  C \,  \nunu^2 \,   \frac{p \log(np)^{3/2} }{\sqrt{n}}  +  5 \nunu^3 \, \frac{  p^{3/2} \log(np) \log(2n) }{\sqrt{n}},
\] for some absolute constant $C$.

\end{theorem}

\noindent \emph{Summary of the proof:} Our approach is based on the optimal transport theory \cite{OttoVillani, BobkovEd, Bonis}. It consists of the following steps. 

\begin{itemize}
 \item A very useful tool in  approximation in probability is the Wasserstein distance (expression \ref{equation_3}). It reveals the joint distribution and minimises the difference between two random variables by increasing their dependence with some fixed marginal distributions.   The multivariate Central Limit Theorem for $W_2$ distance was described by \cite{Bonis}, where the author 
 uses Stein's method and the entropy estimation for the derivative of $W_2$.  Below, basing on the same technique, we obtain Gaussian approximation bound for $W_L$ with explicit dependence on the power $L$ of the cost function and dimension parameter $p$ (Theorem \ref{WLTsubG}), such that 
\[
W_L(\xiv, \gammav) \leq 
\frac{C \, L^{3/2} \nunu^2 \| \Sigma \|^{1/2} }{\log L } \left( \frac{p}{\sqrt{n}} + \frac{L}{n^{1 - 1/L}} \right)  +  \frac{5 L \nunu^3 p^{3/2} \| \Sigma \|^{1/2} \log(2n) }{\sqrt{n}}.
\]
The proof of the last statement is the most difficult part of this work. It starts with derivative estimation of $W_L$.  For that we construct a Markov random process $X_t$ that transits $X_0 = \xiv$ to $X_\infty = \gammav$. From Lemma \ref{was_int_bound} follows that for the random process $X_t$ under smoothness condition it holds  \[
W_L(X_a, X_b) \leq
\int_{a}^{b}  \left\{ \E \left \|  \E \left[ \frac{d}{dt} X_t \bigg | X_t \right] \right \|^L  
\right \}^{1/L} dt .
\]Then in order to derive an upper bound for $\E \left \|  \E \left[ \frac{d}{dt} X_t  | X_t \right] \right \|^L  $ we involve Rosenthal's martingale inequality (Lemma \ref{rose}), convolution with Hermite polynomials (Lemma \ref{boni}), and modified variant of  Stein's method (Theorem \ref{WLT}).    
\item The sub-Gaussian property is required for moments estimation of random variables $\|\xiv_i \|$ (used in Theorem \ref{WLT}) and in Lemma \ref{subg_moments} we have proved the following upper bound for each $k \geq 2$ and $1 \leq i \leq n$
\[
   \E \left \| \Sigma^{-1/2} \xiv_i  \right \|^k \leq \frac{4 \nunu^k}{n^{k/2}} (\sqrt{p} + \sqrt{2k})^{k}.
\]
\item Denote a joint distribution of $\xiv$ and $\gammav$ by $\pi = \pi(\xiv, \gammav)$.   Markov's inequality provides the following statement \[ 
\min_{\pi \in \Pi[\xiv, \gammav]} \P \left(   \normx{\xiv - \gammav} > \Delta \right) \leq  \frac{W_L^L(\xiv, \gammav)} {\Delta^L}, 
\]
where we set $L = \log (np)$ and $\Delta = C(n, p) \| \Sigma \|^{1/2} \, e^{t/\log(np)}$.

\end{itemize}

\begin{remark}
Note that convergence in probability in Central Limit Theorem does not exist, \textit{i.e.}, there is no random vector $ \gammav$, such that $\xiv(n)$ converges to it when $n \to \infty$. It follows from the Kolmogorov's  $0$-$1$ law. However, at the same time, for each $n$ one should take a new Gaussian vector which depends on $n$ and 
\[
 \| \xiv (n) - \gammav (n) \| \overset{p}{\longrightarrow} 0,
\quad 
n \to \infty,
\]
which follows from Theorem \ref{main_res}.
\end{remark}

\section{Approximation in probability}
Having introduced the main definitions in the previous section, we want to find an upper bound for $ | \normx{\xiv}  - \normx{\gammav} |$  in probability. 
Markov's inequality gives the following for some $L \geq 0$:
\[\label{ME}
\P \left ( \left| \normx{\xiv}  - \normx{ \gammav } \right|  > \Delta \right) 
\leq \P \left( \normx{\xiv - \gammav }  > \Delta \right) 
\leq \frac{\E \normx{\xiv - \gammav }^L}{\Delta^L}. 
\]
The differences $ \normx{ \xiv  -  \gammav }$ decrease when the dependence between the random variables increases. So we want to make them as dependent as possible.
Denote by $\pi(\xiv, \gammav)$ a joint distribution of $\xiv$ and $\gammav$. 
One then has to minimize over $\pi$ with a fixed marginals ($\pi \in \Pi[\xiv, \gammav]$) the following Wasserstein distance 
\begin{equation}\label{equation_3}
W_L^L(\xiv, \gammav) = \min_{\pi \in \Pi[\xiv, \gammav]} \left\{ \E  \| \xiv  - \gammav \|^L \right\}.
\end{equation}
\noindent Notation $\pi \in \Pi[\xiv, \gammav]$ means that $\int \pi(\xiv, d \gammav) $ equals the distribution of $\xiv$ and $\int \pi(d \xiv, \gammav) $ equals the distribution of $\gammav$. Accounting to the previous inequalities,
\begin{equation} \label{markov_bound}
\min_{\pi \in \Pi[\xiv, \gammav]} \P \left(  \left| \normx{\xiv}  - \normx{ \gammav } \right|> \Delta \right) \leq  \frac{W_L^L(\xiv, \gammav)} {\Delta^L}. 
\end{equation}
Below, we will study one/multi dimensional cases separately because of the difference in the complexity of the proof.    An accurate solution of this problem is applicable for a lot of different fields of science \cite{NIPS2013_cuturi, Chernozhukov, cpd_boot, cpd_ecg} and thus so important.

\subsection{One dimensional case}

Here we will obtain a simple and close to optimal (up to a logarithmic factor) estimation of the main bound (\ref{markov_bound}) for the random variable $\xi$ with additional assumption on its distribution (\ref{cdf_diff}).  We will generalise and improve this estimation in the next section with much more difficult approach.

\begin{lemma} [\cite{BobkovEd} Proposition 5.1]
 \label{was_1d}
 Let $\xi, \gamma \in \mathcal{P}(\mathbb{R})$ have the cumulative distributions $F$ and $\Phi$ (standard Gaussian cdf) respectively.
Assume that, for some real numbers $\varepsilon > 0$ and $d \geq 1$,
\begin{equation}
\label{cdf_diff}
| F(x) - \Phi(x) | \leq \varepsilon \left(1 + |x|^d \right) e^{-x^2/2}.
\end{equation}Then 
\[
W_L(\xi, \gamma) \leq C(L, d) \varepsilon,
\]
where  one may take $C(L, d)=(CLd)^{3(d+1)/2}$ with some absolute constant $C$.
\end{lemma}

\begin{corollary} \label{prob_approx_indep}
If the random variables $\xi$ and $\gamma$ fulfill condition (\ref{cdf_diff}), then

\begin{eqnarray*}
&&\min_{ \pi \in \Pi[\xi, \gamma] } \P \left( \left|  \xi  -  \gamma \right| > 
e C(\log n, d) \varepsilon \right)  \leq  \frac{1}{n}.
 \end{eqnarray*}
where  $C(\log n, d)$ is defined in the previous lemma.

\end{corollary}

\begin{proof}
Apply Lemma \ref{was_1d} for pair $(\xi, \gamma)$:
\begin{eqnarray*}
&& W_L^L(\xi, \gamma) = \min_{\pi \in \Pi[\xi, \gamma]}\E  |\xi -\gamma |^L \leq C^L(L, d) \varepsilon^L.
\end{eqnarray*}
In order to find a bound for the probability involve expression (\ref{markov_bound}):
\begin{eqnarray*}
&&\min_{ \pi \in \Pi[\xi, \gamma] } \P \left( \left|  \xi  -  \gamma \right| > \Delta \right)  \leq  \left(\frac{ C(L, d) \varepsilon}{\Delta} \right)^L.
 \end{eqnarray*}
 Set $L = \log n$ and $\Delta = e C(L, d) \varepsilon$.

 \end{proof}
 
 \begin{remark}
 It $(2d)^{d/2} < 1 / \varepsilon$ condition (\ref{cdf_diff}) is stronger than sub-Gaussian property (\ref{expgamgm_2}) that we assume in our main Theorem \ref{main_res}. The justification  relies on the following integration heuristic:
 \[
 \varepsilon \int_{-\infty}^{\infty} e^{\lambda x} |x|^d  e^{-x^2/2} dx \leq  \varepsilon \sup_{x} |x|^d  e^{-x^2/4}
\int_{-\infty}^{\infty} e^{\lambda x}  e^{-x^2/4} dx \leq \varepsilon (2d)^{d/2} e^{2 \lambda^2}. 
 \]
 \end{remark}

\subsection{Multidimensional case}

This part involves multivariate Gaussian approximation in the Wasserstein distance. We first resort proving the approximation of derivative of  $W_L(\xiv, \gammav)$, and only then we utilize Markov's inequality (\ref{markov_bound}) in order to obtain the approximation in probability (\ref{markov_bound}) discussed above.
\noindent Here we will use a specified for the Wasserstein distance version of Stein's method \cite{stein_meth} proposed in paper \cite{Bonis}. For that define a smooth transition from  $\xiv$ to $\gammav$ parametrized by $t$:
\begin{equation} \label{tranz_def}
X_t = e^{-t} \xiv + \sqrt{1 - e^{-2t}} \gammav, \quad t \in [0, \infty].
\end{equation}
\begin{remark}
The choice of this particular transition is explained by the fact that Markov process $X_t$ has generator $A f(x) = \nabla^T \Sigma \nabla f(x) - x^T \nabla f(x)$ and Gaussian stationary measure \cite{bakry_hal}. It is also known as Ornstein--Uhlenbeck process.  
\end{remark}

\noindent Further we will require one important property of this process that will help us to find an upper bound for the derivative of $W_L$ by $t$. It was proposed in paper \cite{OttoVillani} that deals with particular case derivative estimation of $W_2$. There exists a function $\mathcal{G}_t(x)$ that effects the transition of  measure $\mu_t(x)$ of $X_t$ in the following way $\forall f \in \mathbb{C}^1$:
\begin{equation} \label{phi_def}
\frac{d}{dt} \int f(x) d \mu_t(x)  =  \int \nabla^T f(x) \nabla \mathcal{G}_t(x) d \mu_t(x),
\quad t \in [0, \infty).
\end{equation}

\noindent  Its gradient has explicit representation 
\[
\nabla \mathcal{G}_t(X_t) = \E \left[ \frac{d}{dt} X_t \bigg | X_t \right],
\]
since 
\[
\lim_{\Delta \to 0} \frac{\E f(X_{t + \Delta}) - \E f(X_{t}) }{\Delta} = \lim_{\Delta \to 0} \E \nabla^T f(X_t) \frac{\E [ X_{t + \Delta} - X_{t} | X_t ] }{\Delta} = \E \nabla^T f(X_t) \E \left[ \frac{d}{dt} X_t \bigg | X_t \right].  
\]
And differentiating equation (\ref{tranz_def}) by time we derive that 

\begin{equation}\label{phi_expl_def}
\nabla \mathcal{G}_t(X_t) = - e^{-t} \E \left( X_0 - \frac{e^{-t}}{\sqrt{1 - e^{-2t}}} X_\infty \bigg| X_t \right).
\end{equation}

\noindent The next result allows us to bound an 
increment the of Wasserstein distance in a small interval of the transition parameter $t$. 

\begin{lemma} \label{was_int_bound} Let $\nabla \mathcal{G}_t(X_t) \in \R^p$, $t \in [0, \infty] $ be uniformly continuous and bounded random process.  For any cost function $c(x-y)$ in Wasserstein distance $W(c)$ and infinitely small time shift $ds$ it holds
\[
W(c)(X_{t}, X_{t + ds})  \leq \E c \big(ds \nabla \mathcal{G}_t  (X_{t})\big)
\]
and moreover if $\exists f: \R \to \R$ such that  $\forall t,s \geq 0$
\begin{equation}\label{Wtriangle}
f\left \{ W(c)(X_0, X_{t+s}) \right\} -  f\left\{W(c)(X_0, X_{t}) \right\} \leq f \{W(c)(X_t, X_{t+s}) \}  
\end{equation}
then for any $a, b \geq 0$
\[
f\{W(c)(X_a, X_b)\} \leq
\int_{a}^{b} \frac{f \left\{ \E c(ds \nabla \mathcal{G}_t)  
\right \}}{ds} dt .
\]
Particularly, if $c(x-y)$ is any norm in power $L$, then 
\[
W_L(X_a, X_b) \leq
\int_{a}^{b}  \left\{ \E \| \nabla \mathcal{G}_t \|^L  
\right \}^{1/L} dt .
\]
\end{lemma}

\begin{proof} Fix an arbitrary point $t$. Define a new mapping $T_{s}$ on the space of the random variable $X_t$ by equation 
\begin{equation}\label{Tdef}
\frac{\partial}{\partial s} T_{s}(x) =   \nabla \mathcal{G}_{t+s} (T_{s}(x)), 
\quad  T_0(x) = x, \quad s \geq 0.
\end{equation}
It turns out that function $x \to T_{s}(x)$ is the push-forward transport mapping, such that $X_{t+s} =  T_{s}(X_t)$ in distribution (see proof of Lemma 2 in \cite{OttoVillani}). In order to be consistent with our notation we provide below another proof of this proposition.  
Note that for an arbitrary functions $h \in C^1$ and $h_s = h(T^{-1}_s)$ it holds 
\[
\frac{d}{ds}h_s(T_s) = \frac{\partial}{\partial s}h_s(T_s) + \nabla^T h_s(T_s)  \frac{\partial}{\partial s} T_{s}  = 0
\]and with differential equation for $T_s$  (\ref{Tdef})  one gets 
\[
\forall x \in \Omega(X_t) : \quad
\frac{\partial}{\partial s}h_s(x) = -  \nabla^T \mathcal{G}_{t+s}(x) \nabla h_s(x).
\]Then involving  property (\ref{phi_def})  leads to 
\[
\frac{d}{ds} \int h_s(x) d \mu_{t+s}(x) 
=  \int \frac{\partial}{\partial s} h_s(x)  + \nabla^T \mathcal{G}_{t+s}(x) \nabla h_s(x) d \mu_{t+s}   = 0.
\]It proves that $T_s^{-1}(X_{t+s}) = X_t$ and subsequently $X_{t+s} = T_s T_s^{-1}(X_{t+s}) = T_s(X_t)$. 
By the definition of Wasserstein distance and the push-forward transport mapping $T_s$  (\ref{Tdef})
\[
W(c)(X_{t}, X_{t + ds}) \leq \E c \big(T_{ds}(X_t) - T_{0}(X_t) \big) = \E c \big(ds \nabla \mathcal{G}_t  (X_{t})\big)
\]and consequently with the triangle inequality assumption (\ref{Wtriangle})  
\[
\frac{d}{dt} f\{ W(c) (X_0, X_t) \} \leq  \frac{1}{ds}f \{ W(c)(X_{t}, X_{t + ds})\} 
\leq \frac{1}{ds}f \{ \E c (ds \nabla \mathcal{G}_t  (X_{t})) \}.
\]
After the integration we obtain the second statement of this lemma. In case $c(x-y) = \| x-y \|^L$ we set $f = \sqrt[L]{}$ and apply equality 
\[
  \E \| ds \nabla \mathcal{G}_t  (X_{t}) \|^L  =  (ds)^L \E  \| \nabla \mathcal{G}_t  (X_{t}) \|^L .
\]
\end{proof}

\noindent Involve a sequence of useful lemmas that will help us to estimate term $\E  \| \nabla \mathcal{G}_t  (X_{t}) \|^L $ from the previous result.

\begin{lemma} [Rosenthal's inequality \cite{rosental}] \label{rose} For all martingales $S_k = \sum_{i=1}^k{\zetav_i}$ and $2 \leq q < \infty$:
\[
\left [ \E \max_{0 \leq k < \infty} \| S_k \|^q  \right]^{1/q} \leq \frac{C q}{\log q} \left[  \E \left(   \sum_{i=1}^\infty \E_{i-1} \| \zetav_i \|^2  \right)^{q/2} + \E \max_{0 \leq i < \infty} \| \zetav_i \|^q \right]^{1/q}.
\]   
By $\E_{i-1}$ we denote conditional expectation on previous values of the martingale $S_0, \ldots,S_{i-1}$ and by definition $\zetav_i$ are increments of the martingale which may be dependent in general case. 
\end{lemma}

\begin{lemma} \label{boni} \cite{Bonis} Let $\gammav \in \ND(0, I)$ and for all $\alpha \in \N^p$: $x_\alpha \in \R^p$ and $H_\alpha(\gammav, \Sigma)$ be the multivariate Hermite polynomials, defined by 
\begin{equation}\label{Hdef2}
H_{\alpha}(\gammav, \Sigma) = (-1)^{| \alpha |} e^{\frac{ \gammav^T \Sigma^{-1} \gammav } {2}} \partial^{\alpha} e^{-\frac{  \gammav^T \Sigma^{-1} \gammav } {2}},
\quad
H_{\alpha}(\gammav) = H_{\alpha}(\gammav, I).
\end{equation}
 Then
\[
\left [ \E \normX{ \sum_\alpha x_\alpha H_\alpha(\gamma)  }^q \right]^{2/q} \leq \sum_\alpha \max (1, q-1)^{|\alpha|} \alpha! \normX{x_\alpha }^2.
\]
\end{lemma}

\begin{corollary} \label{boni_cons} In general case one may use coordinates replacement and property $\nabla_{(Ax)} = A^{-1} \nabla_x$, such that
\[
\left [ \E \normX{ \sum_{\alpha \geq 0} x_\alpha H_\alpha(\gamma, \Sigma)  }^q \right]^{2/q} \leq \sum_{\alpha \geq 0}  \max (1, q-1)^{|\alpha|} \alpha! \normX{ (\Sigma^{-1/2} x)_\alpha }^2.
\]
\end{corollary}

\noindent Combine the previous two lemmas. 

\begin{lemma} \label{Hsum}  Let $\gammav \in \ND(0, \Sigma)$ and for all $\alpha \in \N^p$: $\zetav^\alpha  \in \mathcal{P}(\R^p)$ and $H_\alpha(\gammav, \Sigma)$ be the multivariate Hermite polynomials (\ref{Hdef2}) and $\zetav^\alpha = \sum_i \zetav_i^\alpha$, where $\{ \zetav_i^\alpha\}$ are independent random vectors. Then for~$2 \leq q < \infty$
\begin{align*}
&\left [ \E \normX{ \sum_{i, \alpha}  \zetav_i^\alpha H_\alpha(\gamma, \Sigma)  }^q \right]^{1/q} \\
&\leq  C(q) \left[  \sum_{i, \alpha}  \alpha! (q-1)^{|\alpha|} \E \normX{\Sigma^{-1/2} \zetav_i^\alpha}^2   \right]^{1/2} \\
& +  \left[  \sum_{\alpha}  \alpha! (q-1)^{|\alpha|}  \normX{\Sigma^{-1/2} \E \zetav^\alpha}^2   \right]^{1/2} \\
&+ C(q) \left[ \sum_i \E \left( \sum_{\alpha}  \alpha! (q-1)^{|\alpha|} \normX{\Sigma^{-1/2} \zetav_i^\alpha}^2  \right)^{q/2}  \right]^{1/q}.
\end{align*}
where one may take $C(q) = \frac{C q}{\log q}$ for some absolute constant $C$. 
\end{lemma}

\begin{proof}Without loss of generality assume $\Sigma = I$. Note that for a random variable $\xi$ operator $[\E \xi^q]^{1/q}$ is the probability norm $L_q$. Use notation 
\[
\delta = \left \| \normX{ \sum_{i, \alpha}  \zetav_i^\alpha H_\alpha(\gamma)  } \right \|_{L_q}.
\]
In order to construct a martingale one has to substitute expectation from each $X^{\alpha}$. 
 For that use the triangle inequality 
\[
\delta
\leq \left \| \normX{ \sum_{i, \alpha}  (\zetav_i^\alpha - \E[ \zetav_i^\alpha]) H_\alpha(\gamma)  } \right \|_{L_q}
+  \left \| \normX{ \sum_{i, \alpha}  \E [ \zetav_i^\alpha] H_\alpha(\gamma)  } \right \|_{L_q}.
\]
Without loss of generality assume that $\E[ \zetav_i^\alpha] = 0$, since $\| \zetav_i^\alpha - \E \zetav_i^\alpha \|_{L_q} \leq 2 \| \zetav_i^\alpha \|_{L_q}$.  Apply Lemma~\ref{boni}.
\[
\delta
\leq  \left\| \left(  \sum_{\alpha}  \alpha! (q-1)^{|\alpha|}  \normX{\sum_i \zetav_i^\alpha  }^2   \right)^{1/2}\right\|_{L_q} .
\]
Apply Lemma \ref{rose} for this hybrid norm and bound maximum by sum of components in power $q$.
\[  
\delta/ C(q) \leq  \left[  \sum_{i, \alpha}  \alpha! (q-1)^{|\alpha|} \E \normX{\zetav_i^\alpha}^2   \right]^{1/2} + \left[ \E \max_i \left(  \sum_{\alpha}  \alpha! (q-1)^{|\alpha|}  \normX{\zetav_i^\alpha}^2   \right)^{q/2} \right]^{1/q},
\]where
\[
   \max_i \left(  \sum_{\alpha}  \alpha! (q-1)^{|\alpha|}  \normX{\zetav_i^\alpha}^2   \right)^{q/2}  \leq 
   \sum_i  \left( \sum_{\alpha}  \alpha! (q-1)^{|\alpha|} \normX{\zetav_i^\alpha}^2  \right)^{q/2} . 
\]
\end{proof}

\begin{theorem} \label{WLT} Let $\gammav \in \ND(0,  \Sigma)$ and $\{\xiv_i\}_{i=1}^n$ be independent random vectors with zero mean and variance $\sum_i \E \xiv_i \xiv_i^T = \Sigma$  and with domain in $\R^p$. 
The Wasserstein distance between $\xiv =\sum_i \xiv_i$ and $\gammav$ with $\| x - y \|^L$ cost function, where $L \geq 2$, has the following upper bound 
\[
\frac{  W_L(\xiv, \gammav) }{\| \Sigma \|^{1/2}} \leq 
\frac{C \, L^{3/2} }{\log L } \left( \mu_4^{1/2} + \mu_{2L}^{1/L} \right) 
 + \frac{  \sqrt{2}  e^{1/2} L^{1/2} \mu_2 }{ \sqrt{n}} + \frac{ e^{1/2} L  \mu_3  }{2} \log(2n),
\]
where $C$ is some absolute constant and with $\xiv'_{i}$ independent copies of $\xiv_{i}$ 
\[
\mu_k =  \left\{ \sum_{i=1}^n \E \left \| \Sigma^{-1/2} (\xiv_{i} - \xiv'_{i}) \right \|^{k}\right\}.
\]
\end{theorem}

\begin{proof} 
According to the base procedure of the Stein's method \cite{stein_meth} we should consequently replace random vectors $\xiv_i$  by their independent copies  $\xiv'_{i}$ trowing random index at each step. It allow to estimate change of some function of interest that depends on the transition process $X_t$ (\ref{tranz_def}) under condition that $\mu_t$ is stationary measure. In our case the function is $ \nabla \mathcal{G}_t(X_t)$ (\ref{phi_expl_def}). For a formal description   
involve two additional random vectors with random uniform index $I \in \{1,\ldots,n\}$
\[
\xiv'(t) = \xiv + (\xiv'_{I} - \xiv_{I}) \Ind_I(t),
\]
where
\[
\Ind_I(t) = \Ind \left[ \| \Sigma^{-1/2} ( \xiv'_{I}- \xiv_{I}) \| \leq \left( \frac{ e^{2t}-1 }{L} \right)^{1/2} \right],
\]
and
\[
\tauv(t)  = \E \left [ 
  \frac{n}{2} (\xiv' - \xiv) \left( 1 + \sum_{\alpha \geq 0} \frac{(\xiv' - \xiv)^\alpha H_{\alpha}(\gammav, \Sigma)}{\alpha! (e^{2t} - 1)^{|\alpha| / 2}} \right)
  \bigg | \xiv, \gammav
\right].
\]
We restrict term  $\|  \Sigma^{-1/2} (\xiv'_{I}- \xiv_{I}) \|$ from the upper side in order to prevent irregularity in random vector $\tauv(t)$ when $t \to 0$. Note that for any smooth function $f$ by construction of $\xiv'(t)$ (ref. Lemma 10 in \cite{Bonis})
\[
\E \tauv(t) f(X_t) = n \E  \left(X'_t - X_t \right) \frac{ f(X'_t) + f(X_t) }{2}  = 0,
\]
where $X'_t = e^{-t} \xiv'(t) + \sqrt{1 - e^{-2t}} \gammav$. And as a consequence $\E(\tauv (t)  | X_t) = 0$. So one may note that $\tauv(t)$ shifts process $X_t$ without changing the its measure. Oppositely $ \nabla \mathcal{G}_t(X_t) $ shifts $X_t$ outside its measure towards $\gammav$. On order to reduce the unconditional expectation of $ \nabla \mathcal{G}_t(X_t) $  we will substitute $\tauv(t)$ from it.        
Accounting to the definition of $\nabla \mathcal{G}_t(X_t)$ (\ref{phi_expl_def}) we obtain 
\[
 \nabla \mathcal{G}_t(X_t) =  \nabla \mathcal{G}_t(X_t) -  e^{-t} \E(\tauv (t)  | X_t) = -   e^{-t} \E \left( \xiv - \frac{ 1 }{\sqrt{e^{2t} - 1}} \gammav + \tauv(t) \bigg| X_t \right).
\]We have just added  zero component to the initial value of  $ \nabla \mathcal{G}_t(X_t) $. Below we will use Jensen's inequality moving the conditional expectation outside the norm and $\tauv(t)$  will decrease the expected value of the norm. Add notation $\Sigma_i = \E \xiv_i \xiv_i^T$ with evident property 
$
 \sum_{i=1}^n \Sigma_i = \Sigma.
$
Use a short notation $\E_{X_t}$ for the expectation operator with $X_t$ condition.  Unwrap the last expression with $ \nabla \mathcal{G}_t(X_t)$
\[
 \E \left( \xiv - \frac{ 1 }{\sqrt{e^{2t} - 1}} \gammav + \tauv(t) \bigg| X_t \right)
\]\begin{equation} 
\label{phi_plus_tau}
\begin{split}
& = \sum_{i=1}^n  \E_{X_t}  (\xiv_{i} - \xiv'_{i}) (1 - \Ind_i(t)) \\
 & + \frac{1}{\sqrt{e^{2t} - 1}} \E_{X_t}  \sum_{i=1}^n \left( \frac{(\xiv'_{i} - \xiv_{i}) (\xiv'_{i} - \xiv_{i})^T  \Ind_i(t)   }{2} - \Sigma_i \right) \Sigma^{-1} \gammav \\
 & + \sum_{\alpha \geq 2} \frac{1}{2 \alpha! (e^{2t} - 1)^{|\alpha|/2}} \E_{X_t} H_\alpha(\gammav, \Sigma) \sum_{i=1}^n  (\xiv'_{i} - \xiv_{i}) (\xiv'_{i} - \xiv_{i})^{\alpha}  \Ind_i(t).    
\end{split}
\end{equation}Use notations from Lemma \ref{Hsum} and set 
\[
\zetav_i^0 =   (\xiv_{i} - \xiv'_{i}) (1 - \Ind_i(t)).
\]For $|\alpha| = 1$
\[
\zetav_i^\alpha = \frac{1}{\sqrt{e^{2t} - 1}}  \left( \frac{(\xiv'_{i} - \xiv_{i}) (\xiv'_{i} - \xiv_{i})^\alpha  \Ind_i(t)   }{2} -   \frac{ \E (\xiv'_{i} - \xiv_{i}) (\xiv'_{i} - \xiv_{i})^\alpha    }{2}  \right).
\]
For $|\alpha| > 1$
\[
\zetav_i^{\alpha} =   \frac{1}{2 \alpha! (e^{2t} - 1)^{|\alpha| /2}}  (\xiv'_{i} - \xiv_{i})  (\xiv'_{i} - \xiv_{i})^{\alpha}  \Ind_i(t) .   
\]
Next we want no estimate the power $L$ of $ \sum_{i, \alpha}  \zetav_i^\alpha H_\alpha(\gamma, \Sigma)  $, $\alpha \in \N$. Define a normalized vector 
\[
\overline{\xiv}_i = \Sigma^{-1/2} (\xiv'_{i} - \xiv_{i}).
\] 
In the expressions below we are going to use the following inequalities $\forall \alpha > 0$
\[
\Ind_{i}(t) \leq \frac{1}{\| \overline{\xiv}_i  \|^{| \alpha |} }  \left(\frac{e^{2t} - 1}{L} \right)^{| \alpha | / 2},
\quad 1 - \Ind_{i}(t) \leq \| \overline{\xiv}_i  \|^{| \alpha | } \left(\frac{L}{e^{2t} - 1} \right)^{ | \alpha | / 2}.
\]Note that $\forall l > 0$
\[
\| \zetav_i^0 \| \leq  \| \xiv'_{i} - \xiv_{i} \| \| \overline{\xiv}_i  \|^{l} \left(\frac{L}{e^{2t} - 1} \right)^{l/2},
\]
\[
\sum_{|\alpha| = 1 } \|\zetav_i^\alpha \|^2 \leq  \frac{ \| \xiv'_{i} - \xiv_{i} \|^2 \| \overline{\xiv}_i  \|^2     }{4 (e^{2t} - 1)}  +  \frac{  (\E \| \xiv'_{i} - \xiv_{i} \| \| \overline{\xiv}_i  \|)^2 }{4 (e^{2t} - 1)}, 
\]
\[
\sum_{|\alpha| = m } \| \zetav_i^{\alpha}\|^2 =   \frac{1}{(2 m!)^2 L^{l} (e^{2t} - 1)^{(m-l)}}  \| \xiv'_{i} - \xiv_{i} \|^2   \|  \overline{\xiv}_i \| ^{2 (m - l)}.  
\]Then setting $l = m - 1$ one gets 
\begin{equation*}
\sum_{\alpha \geq 0 }^{\infty}  \alpha ! (L-1)^{|\alpha|} \| \zetav_i^\alpha \|^2 \leq \frac{L}{e^{2t} -1} 
\left( \left(1 + \frac{e}{4} \right) \| \xiv'_{i} - \xiv_{i} \|^2  \|  \overline{\xiv}_i \|^2 + \frac{  (\E \| \xiv'_{i} - \xiv_{i} \|  \|  \overline{\xiv}_i \|)^2 }{4} \right).
\end{equation*}Apply expectation operator and make summation by $i$ 
\begin{equation}
\label{chast1}
 \left[  \sum_{i, \alpha}  \alpha ! (L-1)^{| \alpha |} \E \| \zetav_i^\alpha \|^2   \right]^{1/2} 
 \leq   \frac{(2L)^{1/2}}{(e^{2t} -1)^{1/2}} 
\left(  \sum_{i=1}^n \E \| \xiv'_{i} - \xiv_{i} \|^2  \|  \overline{\xiv}_i \|^2  \right)^{1/2}.
\end{equation}Analogical expression holds in power $L/2$:  
\begin{equation}\label{chast2}
\left[ \sum_i \E \left( \sum_{\alpha \geq 0}  \alpha! (L-1)^{|\alpha|} \| \zetav_i^\alpha \|^2  \right)^{L/2}  \right]^{1/L}
\leq   \frac{(2L)^{1/2}}{(e^{2t} -1)^{1/2}} 
\left(  \sum_{i=1}^n \E \| \xiv'_{i} - \xiv_{i}\|^{L}  \|  \overline{\xiv}_i \|^L  \right)^{1/L}.
\end{equation}For the other part of bound in Lemma \ref{Hsum} one has to evaluate term  $\| \sum_i \E \zetav_i^\alpha \|^2$. Note that 
\[
\sum_\alpha \left \| \sum_i \E \zetav_i^\alpha \right \|^2 = \sum_{i,j} \sum_\alpha  (\E \zetav_i^\alpha)^T \E X_j^\alpha  
\]
and $\forall l > 0$:
\[
\sum_{|\alpha| = 1 }   (\E \zetav_i^\alpha)^T \E X_j^\alpha   \leq  \frac{ \E \| \xiv'_{i} - \xiv_{i} \|  \|  \overline{\xiv}_i \|^{1 + l} \E \| \xiv'_{j} - \xiv_{j} \|  \|  \overline{\xiv}_i \|^{1 + l}     }{4} \frac{L^{l}}{(e^{2t} - 1)^{l+1}}  
\]and for $m >1$:
\[
\sum_{|\alpha| = m }   (\E \zetav_i^\alpha)^T \E X_j^\alpha \leq   \frac{
\E  \| \xiv'_{i} - \xiv_{i} \|  \|  \overline{\xiv}_i \| ^{ (m - l)} \E \| \xiv'_{j} - \xiv_{j} \|  \|  \overline{\xiv}_i \| ^{ (m - l)} 
}{(2 m!)^2 L^{l} (e^{2t} - 1)^{(m-l)}}.  
\]
Then setting $l = m - 2$ one obtains
\begin{equation}\label{chast3}
 \left[  \sum_{\alpha \geq 0}  \alpha! (L-1)^{|\alpha|}  \| \E \zetav^\alpha \|^2   \right]^{1/2} 
 \leq   \frac{e^{1/2} L}{e^{2t} -1} 
  \sum_{i=1}^n \E \| \xiv'_{i} - \xiv_{i} \|  \|  \overline{\xiv}_i \|^2.
 \end{equation}
Moreover, setting $l = m - 1$ we additionally derive that   \begin{equation}\label{chast4}
  \left[  \sum_{\alpha \geq 0}  \alpha! (L-1)^{|\alpha|}  \| \E \zetav^\alpha \|^2   \right]^{1/2} 
 \leq   \frac{e^{1/2} L^{1/2}}{(e^{2t} -1)^{1/2}} 
  \sum_{i=1}^n \E \| \xiv'_{i} - \xiv_{i}\|  \|  \overline{\xiv}_i \| 
  = \frac{ \| \Sigma \|^{1/2} \mu_2 e^{1/2} L^{1/2}}{(e^{2t} -1)^{1/2}}. 
 \end{equation}
 The last term requires two bounds because integration of function  $e^{-t}/(e^{2t} -1)$ diverges near zero point. We just have shown above in equation (\ref{phi_plus_tau}) that 
 \[
  \nabla \mathcal{G}_t(X_t)  = -e^{-t}  \sum_{i, \alpha} \zetav^\alpha_i H_\alpha(\gammav).
 \]
 From Lemma \ref{Hsum} and expressions (\ref{chast1}), (\ref{chast2}), (\ref{chast3}), (\ref{chast4}) follows that 
 \begin{align*}
 \frac{ \left[ \E \| \nabla \mathcal{G}_t(X_t) \|^L  \right]^{1/L} } {\| \Sigma \|^{1/2}} &\leq  \frac{C \, L^{3/2} e^{-t}}{\log(L) (e^{2t} -1)^{1/2}} \left( \mu_4^{1/2}  
 + \mu_{2L}^{1/L} \right) \\
 & + \min \left[  \frac{e^{1/2} L e^{-t}}{e^{2t} -1}  \mu_3, \,\,  \frac{ e^{1/2} L^{1/2} e^{-t}}{(e^{2t} -1)^{1/2}} \mu_2  \right].
 \end{align*}
 Now in order to obtain the bound for $W_L(\xiv, \gammav)$ by means of Lemma \ref{was_int_bound} one has to integrate the previous expression by $t$. Apply the following integrals:
\[
\int_0^{\infty} \frac{e^{-t}}{(e^{2t} -1)^{1/2}}dt  = 1 .
\]For $t_0 = 1/n$

\begin{align*}
& A \int_0^{t_0} \frac{e^{-t}}{(e^{2t} -1)^{1/2}}dt  + B \int_{t_0}^{\infty} \frac{e^{-t}}{e^{2t} -1}dt  \\
& = A(1 - e^{-2t_0})^{1/2}
+ B \left( -e^{-t_0} + \frac{1}{2} \log \frac{e^{-t_0} + 1}{ 1 - e^{-t_0}} \right) \\
& \leq \frac{\sqrt{2} A}{ \sqrt{n}} + \frac{B}{2} \log(2n).
\end{align*}
Set $A =  e^{1/2} L^{1/2} \mu_2$ and $B = e^{1/2} L \mu_3$, then finally
 \[
  \int_0^{\infty} \frac{ \left[ \E \|\nabla \mathcal{G}_t(X_t) \|^L  \right]^{1/L} }{\| \Sigma \|^{1/2}} dt \leq 
 \frac{C \, L^{3/2} }{\log L } \left( \mu_4^{1/2} + \mu_{2L}^{1/L} \right) 
 + \frac{ \sqrt{2}   e^{1/2} L^{1/2}  \mu_2 }{ \sqrt{n}} + \frac{ e^{1/2} L  \mu_3  }{2} \log(2n).
 \]
\end{proof}

\noindent Consider a random vector \(\xiv\) which has restricted exponential or sub-Gaussian moments: $ \exists \; \gm>0$
\begin{equation}
    \log \E \exp\bigl( \gammav^{\T}\xiv\bigr)
    \le
    \| \gammav \|^{2}/2,
    \quad
    \gammav \in \R^{\dimp}, \quad \| \gammav \| \le \gm 
\label{expgamgm}.
\end{equation}
For ease of presentation, assume below that \( \gm \) is sufficiently large.

\begin{lemma}[\cite{spokoiny2017}]
\label{LLbrevelocro}   
Define $\xxc = \gm^{2}/4$.
Let \eqref{expgamgm} hold and 
\( 0.3 \gm \ge \sqrt{\dimp} \).
Then
for each \( \xx > 0 \)
\[
    \P\bigl( \|\xiv\| \ge \zq(\dimp,\xx) \bigr)
    \leq
    2 \ex^{-\xx} + 8.4 \ex^{-\xxc } \Ind(\xx < \xxc) ,
\]   
where \( \zq(\dimp,\xx) \) is defined by
\[
\label{PzzxxpBro}
    \zq(\dimp,\xx)
     =  
    \begin{cases}
        \bigl( \dimp + 2 \sqrt{\dimp \xx} + 2 \xx\bigr)^{1/2}, &  \xx \le \xxc \\
        \gm + 2 \gm^{-1} (\xx - \xxc)   , & \xx > \xxc .
    \end{cases}
\]    
\end{lemma}

\begin{lemma}\label{subg_moments}

With conditions from Lemma \ref{LLbrevelocro} for each $k \geq 2$
\[
   \E \left \| \xiv  \right \|^k \leq 4 (\sqrt{p} + \sqrt{2k})^{k}.
\]
\end{lemma}

\begin{proof} First  separate events $\{  \| \xiv   \| \leq \sqrt{p}  \} $ and $\{  \| \xiv   \| > \sqrt{p}  \} $. Then we obtain
 \[
\E \| \xiv \|^k   \leq  p^{k/2} +  k \int_{\sqrt{p}}^{\infty} t^{k-1} \P (\| \xiv \| > t) dt. 
\]
Apply Lemma \ref{LLbrevelocro} that is valid in the region $\{  t > \sqrt{p}  \} $ 
\begin{align*}
&  \int_{\sqrt{p}}^{\infty} t^{k-1} \P (\| \xiv \| > t) dt \\
& \leq 2 \int_{0}^{\infty} z(p, \xx)^{k-1} (\ex^{-\xx}) z'(p, \xx) d\xx  \\
& + 8.4 e^{-\xxc} \int_{0}^{\xxc} z(p, \xx)^{k-1} d z(p, \xx).
\end{align*}Through integration by parts we find that 
\[
k \int_{0}^{\infty} z(p, \xx)^{k-1} (\ex^{-\xx}) z'(p, \xx) d\xx = -z^k(p, 0) + \int_{0}^{\infty} z(p, \xx)^{k} (\ex^{-\xx})  d\xx .
\]The function $z(p, \xx)$ behaves differently before and after the point $\xxc$.  So we will integrate it separately. 
\begin{align*}
& \int_{0}^{\xxc} z(p, \xx)^{k} (\ex^{-\xx})  d\xx  = \int_{0}^{\xxc} (\sqrt{p} + \sqrt{2x} )^{k} (\ex^{-\xx})  d\xx \\
& \leq \int_{0}^{\infty} (\sqrt{p} + y)^k e^{- \frac{y^2}{4}}  e^{-\frac{y^2}{4}} d \frac{y^2}{2} \\
& \leq \{ \text{set } y = \argmax (\sqrt{p} + y)^k e^{- \frac{y^2}{4}}  \} \leq \frac{  (\sqrt{p} + \sqrt{2k})^k }{2}  \int_{0}^{\infty}  e^{- \frac{y^2}{4}} d \frac{y^2}{2} \\
&  \leq   (\sqrt{p} + \sqrt{2k})^k  .
\end{align*}
\begin{align*}
& \int_{\xxc}^{\infty} z(p, \xx)^{k} (\ex^{-\xx})  d\xx  = \int_{\xxc}^{\infty} \left(g + \frac{2}{g} (\xx - \xxc) \right)^{k} (\ex^{-\xx})  d\xx \\
& = e^{-\xxc} \int_{0}^{\infty}   \left(g + \frac{2}{g} y \right)^{k} e^{-\frac{y}{2}} e^{-\frac{y}{2}} dy \\
& \leq \left \{  \left(g + \frac{2}{g} y \right)^{k} e^{-\frac{y}{2} } \leq \max  \left \{g, \frac{4k}{g}  \right\}^{k} \right \} \leq  2 e^{-\xxc}  \max \left\{ g,   \frac{4k}{g}  \right\}^k  \\
& \leq 2 e^{-k/2} (k)^{k/2}.
\end{align*}Compute the last integral part: 
\[
8.4 e^{-\xxc} k  \int_{0}^{\xxc} z(p, \xx)^{k-1} d z(p, \xx)  = 8.4 e^{-\xxc} (\sqrt{p} + \sqrt{2\xxc})^k \leq (\sqrt{p} + \sqrt{k})^k.
\]


\end{proof}

\noindent Using the last lemma one may evaluate moments in Theorem \ref{WLT}. 

\begin{theorem} \label{WLTsubG} Under conditions from Theorem \ref{WLT} and additional condition that $ \forall i: \, \frac{\sqrt{n}}{\nunu} \Sigma^{-1/2} \xiv_i $ is sub-Gaussian (\ref{expgamgm})

\[
W_L(\xiv, \gammav) \leq 
\frac{C \, L^{3/2} \nunu^2 \| \Sigma \|^{1/2} }{\log L } \left( \frac{p}{\sqrt{n}} + \frac{L}{n^{1 - 1/L}} \right)  +  \frac{5 L \nunu^3 p^{3/2} \| \Sigma \|^{1/2} \log(2n) }{\sqrt{n}},
\] for some absolute constant $C$.

\end{theorem}

\begin{remark} The main advantage of this theorem is explicit and sharp dependence on parameters $n$, $p$ and $L$.  
 A simplified asymptotic of this upper bound is 
\[
O\left( \frac{ L^{3/2} p + L p^{3/2} \log n }{\sqrt{n}} \right).
\]
Compare the last result with Lemma \ref{was_1d} in one dimensional case ($p=1$). The asymptotic by $L$ is the same or better depending on the distribution of $\xi$.
This bound remains sharp in case $L = 1$ and comparable with results of \cite{Bentkus, buzun2019gaussian} that studies Gaussian approximation in $W_1$.  
\end{remark}

\section{Approximation in distribution}

In the previous section we derived approximation in probability. Now, we will combine Theorem \ref{WLTsubG} (keeping in mind inequality \ref{markov_bound}) with anti-concentration property (Lemma \ref{anti-concentration}) and will prove approximation in distribution.
\begin{lemma}\label{anti-concentration}\cite{chernozhukov2017detailed}.
Let $\gammav$ be a centered Gaussian random vector in $\R^p$ such that $\E[\gammav^2_k]\geq \underline{\sigma}^2$  for all $k \in \{1,\ldots,p\}$ and some constants $\underline{\sigma}^2>0$, $\Delta>0$, then
\begin{eqnarray*}
&&\P\left( \max_{1 \leq k \leq p} \gamma_k \leq t+\Delta \right)-\P \left(  \max_{1 \leq k \leq p} \gamma _k \leq t \right ) \leq \Delta C_A(\log p),\\
&&C_A(\log p )=\frac{1}{\underline{\sigma}}(\sqrt{2\log p }+2)
.\end{eqnarray*}

\end{lemma}

\begin{lemma}\cite{fang2020highdimensional} \label{g_cmp}
Let $\gammav \in \ND(0, \Sigma)$, $\gammav' \in \ND(0, \Sigma')$ be centered Gaussian random vectors in $\R^p$, such that $\overline{\sigma}^2 \geq \E[\gammav^2_k]\geq \underline{\sigma}^2$ for all $k \in \{1,\ldots,p\}$. Then 
\[
\max_t \left | 
 \P\left( \max_{1 \leq k \leq p} \gamma_k \leq t \right)-\P \left(  \max_{1 \leq k \leq p} \gamma' _k \leq t \right ) \right |
\leq 
C \delta \log p \, 
\log \left( \frac{ \overline{\sigma}  } { \delta \underline{\sigma}   } \right), 
\]
where $C$ is some absolute constant and
\[
 \delta =  \frac{  \| \Sigma - \Sigma' \|_{\infty} } { \lambda_{\min}(\Sigma) }.
\]
\end{lemma}

\noindent Let's introduce an additional lemma which is usually needed to prove that convergence in distribution follows from the convergence in probability.
\begin{lemma}\label{weak_convegence}
For random variables $\xi$ and $\eta$, and a shift of size $\Delta>0$, 
\[
\P(\eta < t - \Delta) - \P(|\xi| \geq \Delta ) \leq \P(\eta + \xi < t) \leq  \P(\eta < t + \Delta) + \P(|\xi| \geq \Delta )
.\]
\end{lemma}
\begin{proof}
Consider two events $H_1=\left\{|\xi|\geq\Delta\right\}$ and $H_2=\left\{|\xi|<\Delta\right\}$, they  compose a full group:
\begin{eqnarray*}
&&\P(\xi+\eta<t)=\P(\xi+\eta<t,H_1)+\P(\xi+\eta<t,H_2)=\P_1+\P_2
.\end{eqnarray*}
Lets find bounds for $\P_1+\P_2$. For $\P_1$, the bounds are
\begin{eqnarray*}
&&0\leq \P_1\leq \P(H_1)
.\end{eqnarray*}
For $\P_2$, the upper bound will be
\begin{eqnarray*}
&&\P_2=\P(\xi+\eta<t,-\Delta<\xi<\Delta)\leq \P(-\Delta + \eta <t)=\P(\eta<t+\Delta)
,\end{eqnarray*}
and the lower bound
\begin{eqnarray*}
&&\P_2=\P(\xi+\eta<t,-\Delta<\xi<\Delta)\geq\P(\Delta+\eta<t,-\Delta<\xi<\Delta)\geq\\
&&\P(\Delta+\eta<t)-\P(|\xi|\geq\Delta)=\P(\eta<t-\Delta)-\P(H_1)
.\end{eqnarray*}
We have found the bounds for $\P_1+ \P_2$ and consequently   for $\P(\xi+\eta < t )$.

\end{proof}

\begin{theorem} \label{final_dist_approx}
   Let $\gammav \in \ND(0, \Sigma)$ and $\{\xiv_i\}_{i=1}^n$ be i.i.d random vectors with zero mean and $\frac{1}{n} \Sigma$ variance and with domain in $\R^p$.  Assume additionally that each random variable $ \frac{\sqrt{n}}{\nunu} \xi_{ik} $ is sub-Gaussian (ref. expression \ref{expgamgm}). Restrict the diagonal elements of the variance matrix such that $\forall i: \underline{\sigma} \leq \Sigma_{ii} \leq \overline{\sigma}$. Denote by $\lambda_{\min}$ the minimal eigenvalue of matrix $\Sigma$. Then $\forall t \in \R $
   
   \[
 \left | 
   \P \left(\max_k \xi_k<t \right)  -  \P \left(\max_k\gamma_k<t \right)
 \right |
 \leq O \left( 
 \nunu^3 \frac{ \overline{\sigma} }{ \underline{\sigma} } \frac{\log^2 (np) }{ \sqrt{n}}
 + \nunu^3 \frac{\overline{\sigma}^2}{\lambda_{\min}} \frac{  \log(np)^{5/2} \log n }{  \sqrt{n}}
  \right).
    \]

\end{theorem}
\begin{proof}
Consider a sequence of Gaussian random variables $\gamma'_1,\ldots,\gamma'_p$, such that all elements  are distributed as elements of vector $\gammav$, but their variance matrix $\Sigma'$ is unknown. We will approximate in probability each random variable  $\xi_k$ by $\gamma'_k$ independently from the other components. Next we will estimate $\| \Sigma - \Sigma' \|_\infty$ from the upper side and, finally, by means of Lemma \ref{g_cmp} we will compare distributions of $\gamma'$ and the original $\gamma$.    
From the Lemma \ref{weak_convegence}, one gets that for an arbitrary shift $\Delta > 0$
\begin{eqnarray*}
&&\P \left (\max_k\gamma'_k<t-\Delta \right)-\P \left (|\max_k\gamma'_k-\max_k\xi_k|\geq \Delta \right)\leq \P \left( \max_k\xi_k<\Delta \right)\\
&&\leq\P\left(\max_k\gamma'_k<t+\Delta\right)+\P\left(|\max_k\gamma'_k-\max_k\xi_k|>\Delta \right).
\end{eqnarray*}
Next, we use anti-concentration (Lemma~\ref{anti-concentration}) that allows us to ``move'' $\Delta$ out of the probability function:
\begin{eqnarray*}
&&\P \left(\max_k\gamma'_k<t+\Delta \right)\leq \P \left(\max_k\gamma'_k<t \right)+\Delta C_A, \\
&&\P \left(\max_k\gamma'_k<t-\Delta \right)\geq \P \left(\max_k\gamma'_k<t \right) - \Delta C_A.
\end{eqnarray*}
Now, we have to minimize the following expression by $\Delta$:
\begin{eqnarray*}
&& \min_{\pi \in \Pi[\gamma_1',\ldots, \gamma_p',\xiv]} \P\left(|\max_k\gamma'_k-\max_k\xi_k|>\Delta \right)+\Delta C_A,
\end{eqnarray*}
where  $\pi$ is a common distribution of $\gammav'$ and $\xiv$ and notation  $\pi \in \Pi[\gamma_1',\ldots, \gamma_p',\xiv]$ means that  the distribution of each gamma component and  $\xiv$ are fixed, but it does not include the common distribution of $\gammav$ as opposed to $\Pi[\gammav',\xiv]$.   
We will find the bound for $\P(|\max_k\gamma'_k-\max_k\xi_k|>\Delta)$ combining Boole's  inequality and one dimensional variant of equation (\ref{markov_bound}). 
\begin{align*}
\min_{\pi \in \Pi[\gamma_1',\ldots, \gamma_p',\xiv]} \P\left(|\max_k\gamma'_k-\max_k\xi_k|>\Delta \right) 
&\leq \min_{\pi \in \Pi[\gamma_1',\ldots, \gamma_p',\xiv]} \left\{ \sum_{k=1}^p \P(| \gamma'_k- \xi_k|>\Delta) \right\} \\
&= \sum_{k=1}^p \min_{\pi_k \in \Pi[\gamma_k',\xi_k]} \P(| \gamma'_k- \xi_k|>\Delta).
\end{align*}
In the last step we set the argmin $\pi = \pi(\xiv) \pi_1(\gamma_1' | \xi_1) \ldots \pi_p(\gamma_p' | \xi_p)$ and this allowed us to swap the sum and min operations.  From (\ref{markov_bound}) we obtain that $\forall k$
\[
\min_{\pi_k \in \Pi[\gamma_k',\xi_k]} \P(| \gamma'_k- \xi_k|>\Delta) \leq \frac{1}{\Delta^L} \sum_k  W_L^L(\xi_k, \gamma'_k).
\]
Note that  for a fixed component $k$:  $ W_L(\xi_k, \gamma'_k) = \Sigma_{kk}  W_L(\xi_k/ \Sigma_{kk}, \gamma'_k/ \Sigma_{kk}) $ and from Theorem \ref{WLTsubG} follows  
\begin{equation}
\label{WL_bound}
W_L(\xi_k/ \Sigma_{kk}, \gamma'_k/ \Sigma_{kk})  \leq   \frac{C L^{3/2} \nunu^3}{\log L \sqrt{n}}.
\end{equation}
Therefore, the minimization task is
\[
  \Delta C_A + \frac{ p}{\Delta^L} \left (  \frac{C \overline{\sigma} L^{3/2} \nunu^3}{\log L \sqrt{n}}\right)^L.
\]
Choose $L  = \log (np)$ and $\Delta$ that will satisfy the condition
\begin{eqnarray*}
&& \frac{ p}{\Delta^L} \left (  \frac{C \overline{\sigma} L^{3/2} \nunu^3}{\log L \sqrt{n}}\right)^L \leq \frac{1}{n},\\
&&\Delta=  \frac{C e \overline{\sigma}\log^{3/2} (np) \nunu^3}{ \sqrt{n}} .
\end{eqnarray*}
Setting this value in the minimization task, gives the required bound in the approximation in probability. So, we have 
   \[
 \left | 
   \P \left(\max_k \xi_k<t \right)  -  \P \left(\max_k\gamma'_k<t \right)
 \right |
 \leq O \left( 
 \nunu^3 \frac{ \overline{\sigma} }{ \underline{\sigma} } \frac{\log^2 (np) }{ \sqrt{n}}
 \right).
    \]
Now compare $\gammav'$ with $\gammav$. For that one has to estimate $\| \Sigma - \Sigma' \|_\infty$. In the next expression we will consider arbitrary element $\Sigma'_{kl}$ and imply property $W_2(\cdot, \cdot) \leq W_L(\cdot, \cdot)$ (Jensen's inequality) and Cauchy–-Bunyakovsky inequality for the correlation of two random variables.  
\begin{align*}
\left | \E \xi_k \xi_l  - \E \gamma'_k \gamma'_l \right|  
& = \left | \E (\xi_k \xi_l - \xi_k \gamma'_l + \xi_k \gamma'_l -  \gamma'_k \gamma'_l) \right| \\
& \leq  \sqrt{\E \xi^2_k} \sqrt{ \E (\gamma'_l - \xi_l)^2 } + \sqrt{\E (\gamma'_l)^2} \sqrt{ \E (\gamma'_k - \xi_k)^2 } \\
& \leq 2  \overline{\sigma}^2 W_L (\xi_k/ \Sigma_{kk}, \gamma'_k/ \Sigma_{kk}) \\
& \leq \frac{C \overline{\sigma}^2 L^{3/2} \nunu^3}{\log L \sqrt{n}}.
\end{align*}
In the last inequality we have used the upper bound (\ref{WL_bound}). 
Denote $\delta_\Sigma = \max_{kl} \left | \E \xi_k \xi_l  - \E \gamma'_k \gamma'_l \right| / \lambda_{\min} (\Sigma) $. From Lemma \ref{g_cmp} follows $\forall t$
\begin{align*}
 \left | 
 \P\left( \max_{1 \leq k \leq p} \gamma_k \leq t \right)-\P \left(  \max_{1 \leq k \leq p} \gamma' _k \leq t \right ) \right |
&\leq 
C \delta_\Sigma  \log p \, 
\log \left( \frac{ \overline{\sigma}  } { \delta_\Sigma \underline{\sigma}   } \right) \\
&\leq  \frac{C \overline{\sigma}^2 \log(np)^{5/2} \nunu^3}{ \lambda_{\min} (\Sigma) \sqrt{n}} \log n.
\end{align*}

\end{proof}

\section{Experiments }

We have performed empirical validation of the asymptotic of Gaussian approximation. It includes two experiments: with Wasserstein distance $W_L$ and approximation in distribution with maximum norm. 

Theorem \ref{WLTsubG} states that the convergence rate of  Gaussian approximation for $W_L$ is $O(L^{3/2}p / \sqrt{n} + L p^{3/2} / \sqrt{n} )$. This is relatively close to the empirical asymptotic which is $O(p^{3/2} L^{1/2} / \sqrt{n})$. More precisely, the computed power of $n$ is $0.498 \approx 0.5$,  power of $p$ is $1.506 \approx 1.5$ and power of $L$ is $0.439 \approx 0.5$. The measurements of this experiment are presented in Figure \ref{fig:was}. 

The second experiment relates to Theorem \ref{final_dist_approx}. It states that the convergence rate of the  Gaussian approximation for maximum norm is $O(\log^{5/2}p / \sqrt{n})$. This is also relatively close to the empirical asymptotic which is $O(\log p / \sqrt{n})$. More precisely, the computed power of $n$ is $0.496 \approx 0.5$,  power of $p$ is $0.997 \approx 1.0$. It seems that the logarithmic factor may be improved in this theorem and it also depends on the power of $L$ in the previous asymptotic.  The measurements of this experiment are presented in Figure \ref{fig:max}. 

In both experiments we have sampled random variables $\xi_{ik}$ for $1 \leq i \leq n$ and $1 \leq k \leq p$ from Bernoulli distribution. Amount of samples for $W_L$ is $10$k and for the maximum norm is $100$k. The simulation uses random generators from pytorch\footnote{\url{https://pytorch.org/}} library and GPU device of Tesla V100. For Wasserstein distance computation we have used Sinkhorn \cite{NIPS2013_cuturi} algorithm implementation from library geomloss\footnote{\url{https://www.kernel-operations.io/geomloss/}} with hyperparameters $\text{blur}=0.01$ and $\text{scaling} = 0.99$.

\begin{figure}[hbt!]
\includegraphics[scale=0.3]{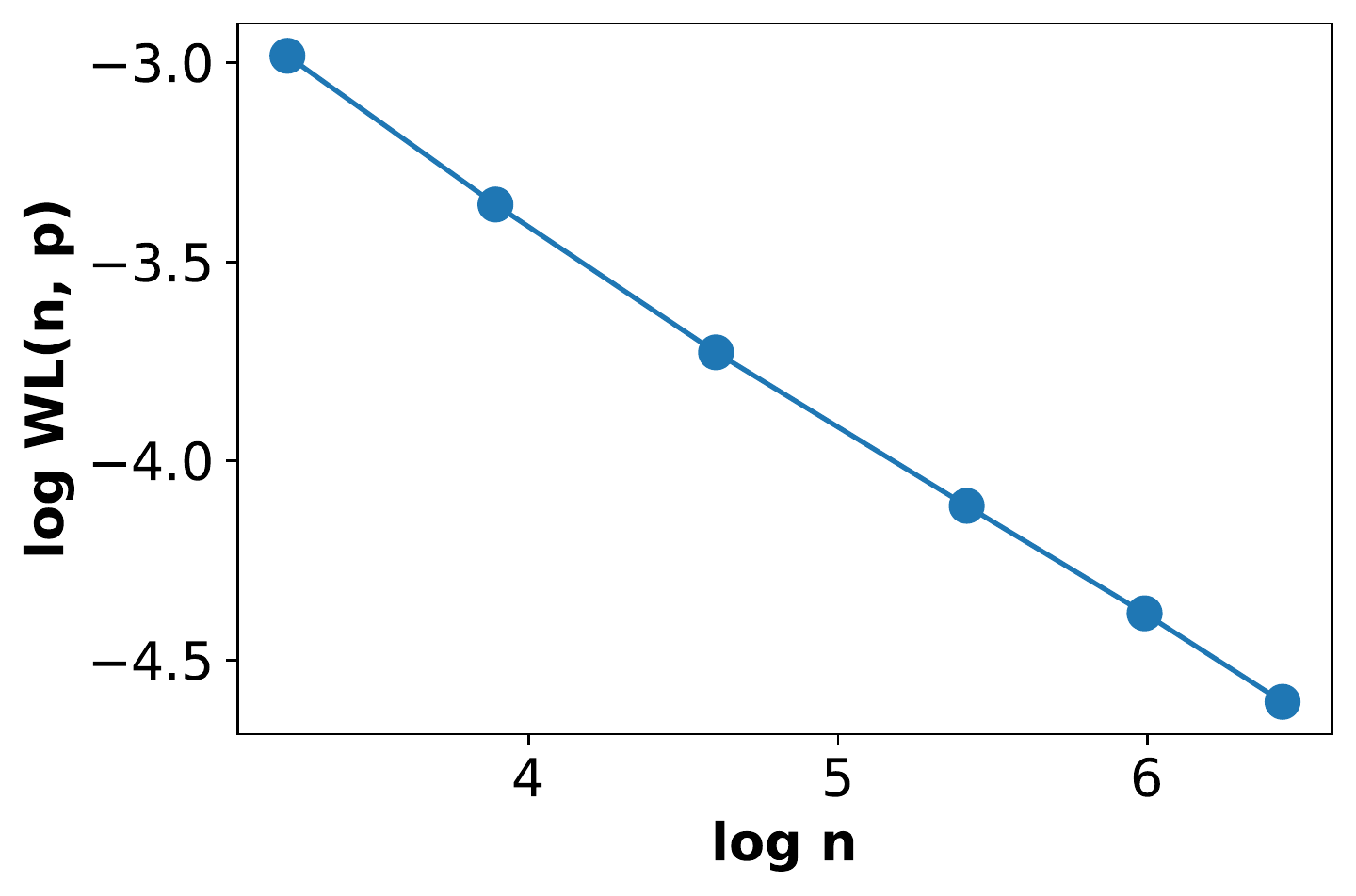}
\includegraphics[scale=0.3]{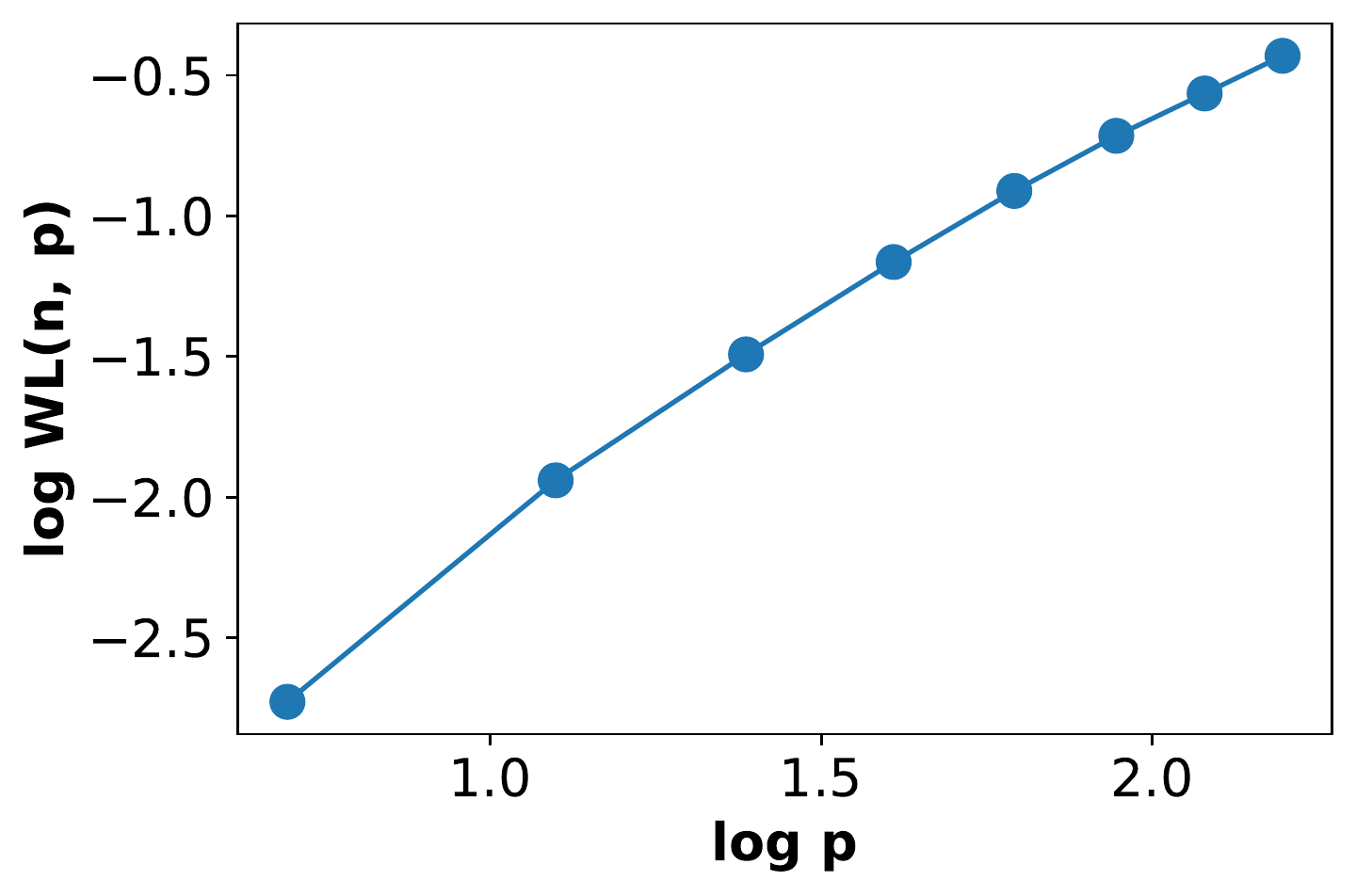}
\includegraphics[scale=0.3]{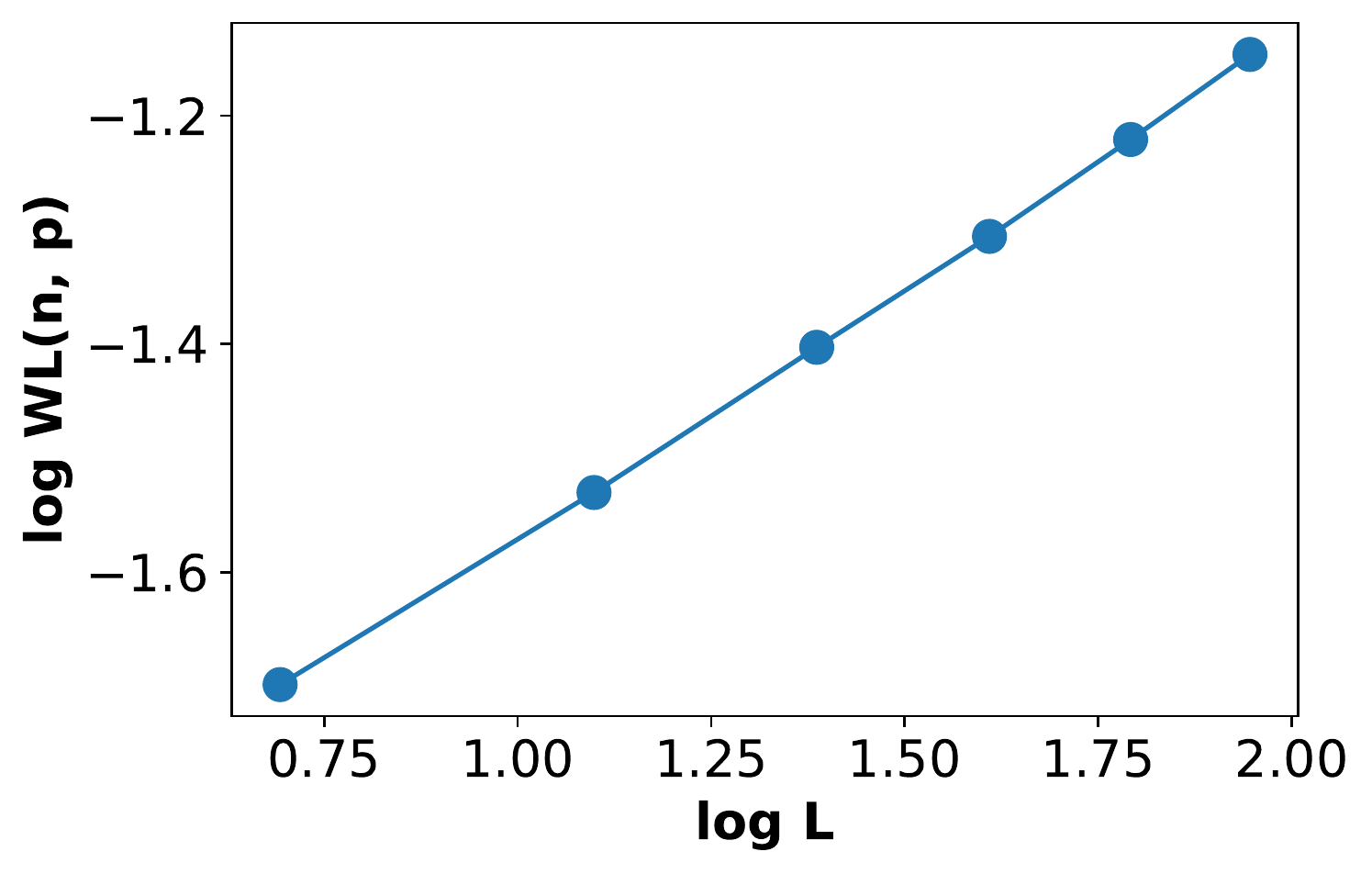}
\caption{Asymptotic dependence of Wasserstein distance $W_L(\xiv, \gammav)$ on parameters $n$ (amount of elements in sum), $p$ (dimension) and $L$ (power of the ground distance). The empirical convergence rate is $O(p^{3/2} L^{1/2} / n^{1/2})$ for i.i.d  $\xi_{ik}\in (\mathcal{B}e(0.5) - 0.5) /\sqrt{n} $. }
\label{fig:was}
\end{figure}

\begin{figure}[hbt!]
\includegraphics[scale=0.25]{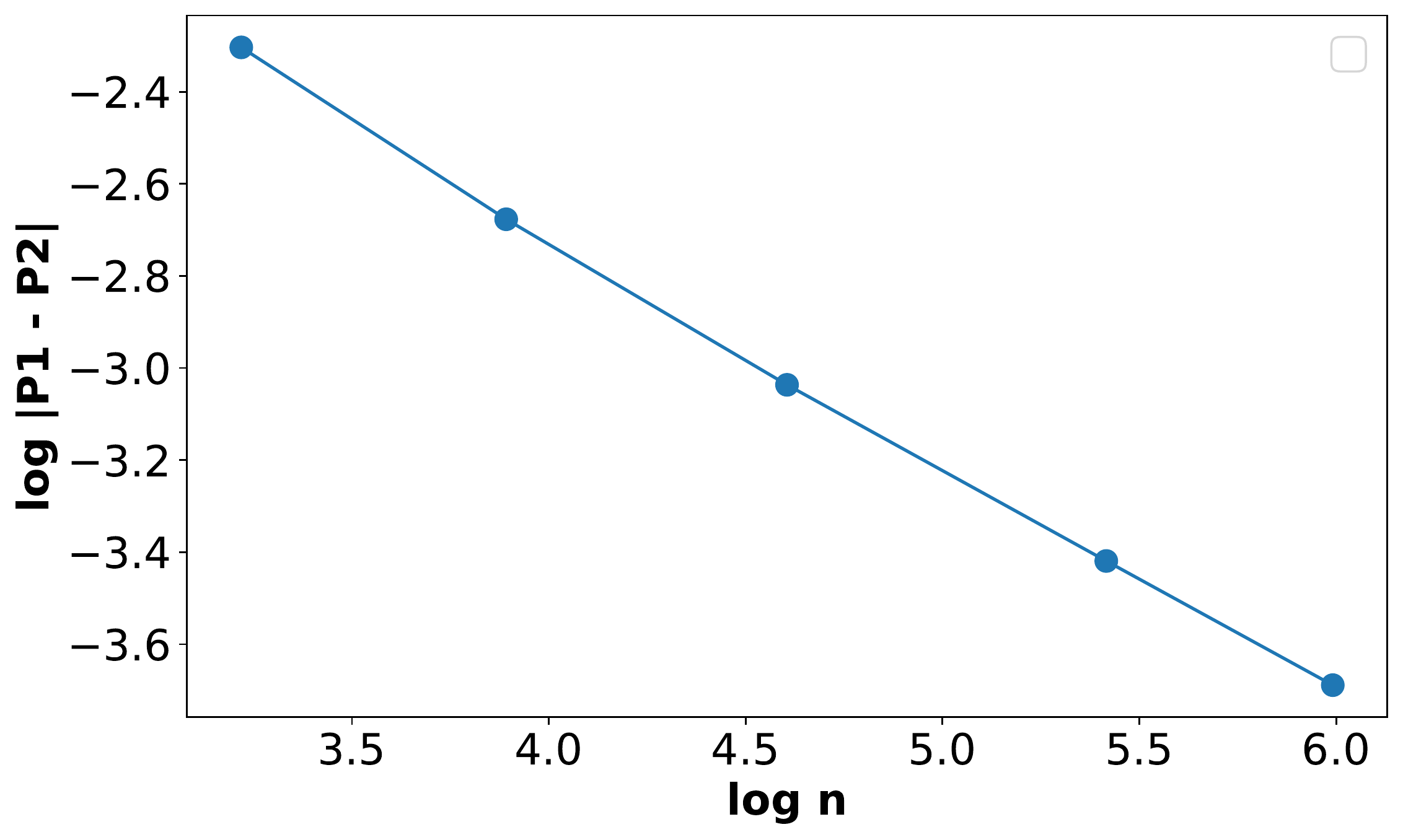}
\includegraphics[scale=0.25]{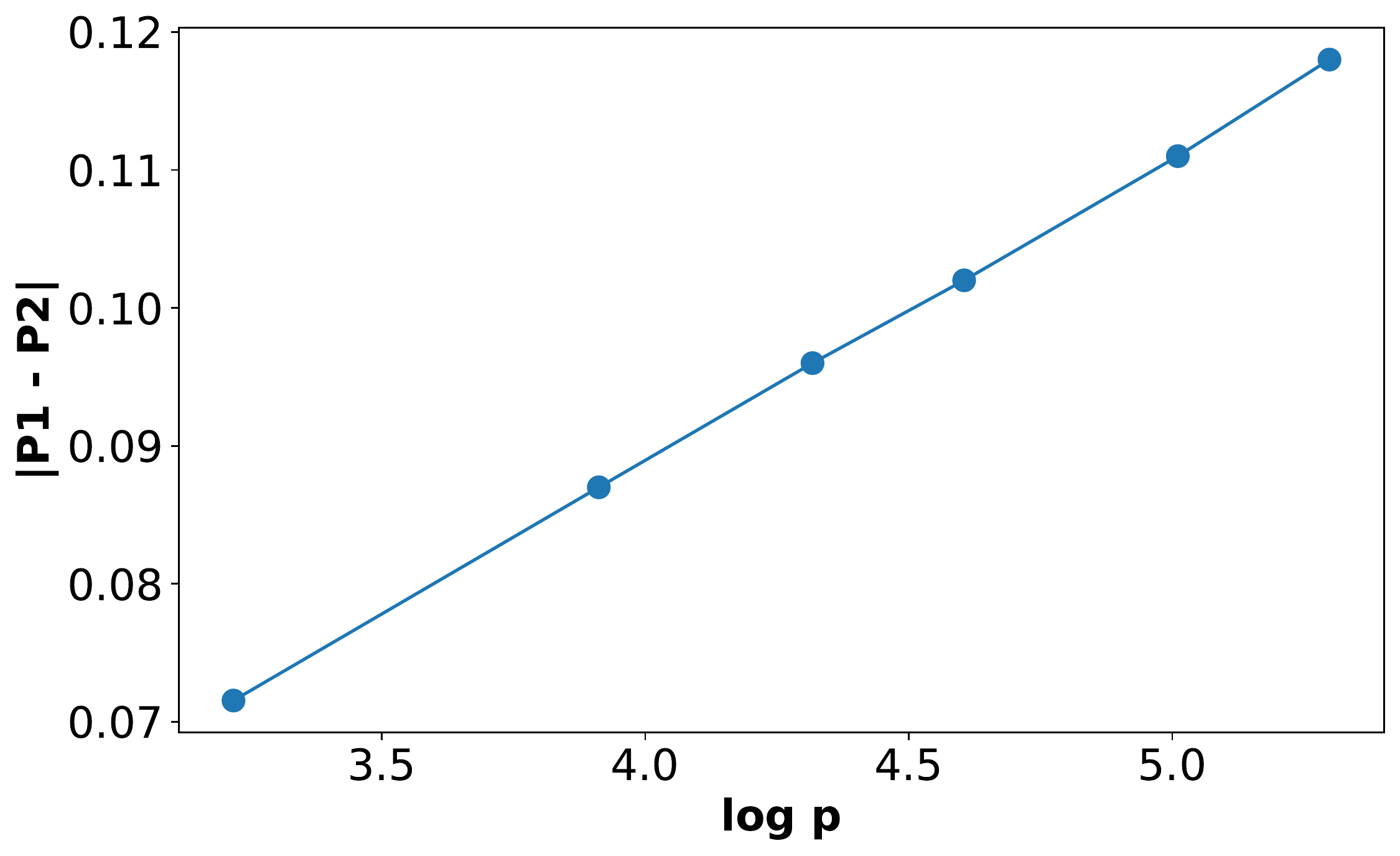}
\caption{
Illustration of the dependence of convergence rate on $\{p,n\}$ of Gaussian approximation with maximum norm. 
Left: $p = 5$, $n \in \{25, 49, 100, 225, 400\}$. 
Right: $n = 50$, $p \in \{25, 50, 75, 100, 150, 200\}$.
The empirical convergence rate is  $O (\log p / \sqrt{n})$ for
 i.i.d  $\xi_{ik}\in (\mathcal{B}e(0.5) - 0.5) /\sqrt{n} $.
}
\label{fig:max}
\end{figure}

\newpage

\bibliographystyle{imsart-number} 

\bibliography{references}

\end{document}